\documentclass{amsart}
\usepackage{amssymb,amsmath,array,multirow,makecell,blindtext,amsthm,enumitem,mathtools,amscd,tikz-cd,amsrefs,dsfont,hyperref,url,caption,float,placeins,bm}
\usepackage[font=small, labelfont=bf]{caption}
\hypersetup{
    colorlinks=true,
    linkcolor=blue,
    citecolor=black,
    filecolor=magenta,      
    urlcolor=cyan,
}
\usepackage[all]{xy}
\usepackage{graphicx}
\usetikzlibrary{positioning}
\usepackage{caption}

 \DeclareFontFamily{U}{wncy}{}
    \DeclareFontShape{U}{wncy}{m}{n}{<->wncyr10}{}
    \DeclareSymbolFont{mcy}{U}{wncy}{m}{n}
    \DeclareMathSymbol{\Sha}{\mathord}{mcy}{"58}

\theoremstyle{plain}
\newtheorem{theorem}{Theorem}[section]
\newtheorem{corollary}[theorem]{Corollary}
\newtheorem{lemma}[theorem]{Lemma}
\newtheorem{remark}[theorem]{Remark}
\newtheorem{proposition}[theorem]{Proposition}

\newcommand{\Z}{\mathbb{Z}}
\newcommand{\Q}{\mathbb{Q}}
\newcommand{\Qcyc}{\mathbb{Q}_{\text{cyc}}}
\newcommand{\R}{\mathbb{R}}
\newcommand{\C}{\mathbb{C}}

\newcommand{\F}{\mathbb{F}}

\newcommand{\Oo}{\mathcal{O}}

\newcommand{\Gal}{\operatorname{Gal}}
\newcommand{\Sel}{\operatorname{Sel}}

\newcommand{\sgn}{\operatorname{sgn}}

\newcommand{\ord}{\operatorname{ord}}

\newcommand{\nf}{\normalfont}
\newcommand{\an}{\text{\nf{an}}}

\newcommand{\e}{\varepsilon}

\newcommand{\Mod}[1]{\ \mathrm{mod}\ #1}

\newcolumntype{?}{!{\vrule width 1pt}}
\newcommand{\lb}{[\![}
\newcommand{\rb}{]\!]}

\title[A bound on the $\mu$-invariants of supersingular elliptic curves]{A bound on the $\mu$-invariants of supersingular elliptic curves}

\author[R.~Gajek-Leonard]{Rylan Gajek-Leonard}
\address[]{Department of Mathematics\\
Union College\\
Bailey Hall 206B\\
Schenectady, NY 12308\\
USA}
\email{gajekler@union.edu}


\begin{document}

\begin{abstract}
Let $E/\Q$ be an elliptic curve and let $p$ be a prime of good supersingular reduction. Attached to $E$ are pairs of Iwasawa invariants $\mu_p^\pm$ and $\lambda_p^\pm$ which encode arithmetic properties of $E$ along the cyclotomic $\Z_p$-extension of $\Q$. A well-known conjecture of B. Perrin-Riou and R. Pollack asserts that $\mu_p^\pm=0$.   We provide support for this conjecture by proving that for any $\ell\geq 0$, we have  $\mu_p^\pm\leq 1$ for all but finitely many primes $p$ with $\lambda_p^\pm=\ell$. Assuming a recent conjecture of D. Kundu and A. Ray, our result implies that $\mu_p^\pm\leq 1$ holds on a density 1 set of good supersingular primes for $E$.  
\end{abstract}

\maketitle

\section{Introduction}

Let $E/\Q$ be an elliptic curve and fix a prime $p$ of good reduction. 
Attached to $E$ is the $p$-primary Selmer group $\Sel(E/\Qcyc)$, where $\Qcyc$ denotes the cyclotomic $\Z_p$-extension of $\Q$. This  group fits into the exact sequence
\begin{equation}\label{exact}
0\rightarrow E(\Qcyc)\otimes \Q_p/\Z_p\rightarrow \Sel(E/\Qcyc)\rightarrow \Sha(E/\Qcyc)\rightarrow 0,
\end{equation}
where $\Sha$ denotes the $p$-part of the Shafarevich-Tate group, 
and therefore encodes many arithmetic properties of $E$ along the cyclotomic line. 

If $p$ is a prime of ordinary reduction then $\Sel(E/\Qcyc)$ is cotorsion  as an Iwasawa module (see \cite[Theorem 17.4]{Kato}) and its characteristic ideal is therefore generated by a polynomial $L_p^{\text{alg}}\in \Z_p[T]$. The algebraic Iwasawa invariants $\lambda_p^{\text{alg}}$ and $\mu_p^{\text{alg}}$ measure the degree and $p$-divisibility of $L_p^{\text{alg}}$, respectively. 
If $E[p]$ is irreducible as a $\Gal(\bar\Q/\Q)$-module, a well-known conjecture of Greenberg \cite[Conjecture 1.11]{Greenberg99} asserts that $\mu_p^{\text{alg}}=0$.

If $p$ is a prime of supersingular reduction then $\Sel(E/\Qcyc)$ is no longer cotorsion, however Kobayashi \cite{Kobayashi} introduced signed Selmer groups $\Sel^\pm(E/\Qcyc)$ which are cotorsion and encode analogous arithmetic data. In particular, the characteristic ideals of $\Sel^\pm(E/\Qcyc)$ are generated by polynomials $L_{p,\text{alg}}^{\pm}\in \Z_p[T]$ which have associated pairs of Iwasawa invariants $\mu_{p,\text{alg}}^\pm$ and $\lambda_{p,\text{alg}}^{\pm}$. In the supersingular setting, $E[p]$ is automatically irreducible and it is similarly conjectured (see \cite[Conjecture 6.3]{PollackThesis} and \cite[Conjecture 7.1]{PR}) that $\mu_{p,\text{alg}}^\pm=0$. 

Recently, Chakravarthy \cite[Theorem 1.3]{Chakravarthy} made progress towards Greenberg's conjecture by proving that $\mu_p^{\text{alg}}\leq 1$ for all but finitely many primes of good ordinary reduction. 
In this article, we prove a similar result in the supersingular setting. 

\begin{theorem}\label{main} Let $\ell\geq 0$ and $*\in \{+,-\}$. Then $\mu_{p,\an}^*\leq 1$ for all but finitely many good supersingular primes $p$ with $\lambda_{p,\an}^*=\ell$.
\end{theorem}
 
 In the above theorem, $\lambda_{p,\text{an}}^\pm$ and $\mu_{p,\text{an}}^\pm$ denote the Iwasawa invariants attached to the \emph{analytic} $p$-adic $L$-functions $L_{p,\text{an}}^\pm \in \Z_p\lb T\rb$ defined by Pollack in \cite{PollackThesis}.
The construction of these $p$-adic $L$-functions requires $a_p=0$, which is automatically true when $p> 3$. The main conjecture of Iwasawa theory in this setting asserts that $L_{p,\text{an}}^\pm$ and $L_{p,\text{alg}}^\pm$ generate the same ideal in $\Z_p\lb T\rb$, and in particular that 
\begin{equation}\label{mcit}
\lambda_{p,\text{alg}}^{\pm}=\lambda_{p,\text{an}}^{\pm}\qquad\text{and}\qquad \mu_{p,\text{alg}}^{\pm}=\mu_{p,\text{an}}^{\pm}.
\end{equation}
Thus, Theorem \ref{main} provides support for the vanishing of $\mu_{p,\text{alg}}^\pm$. 

\begin{remark}\nf The main conjecture is known to hold in many cases: the CM case was established by Pollack and Rubin \cite{PollackRubin}, and Kobayashi \cite[Theorem 1.3]{Kobayashi} proved the containment $(L_{p,\text{an}}^\pm)\subseteq (L_{p,\text{alg}}^\pm)$ for non CM curves. A proof of the full supersingular main conjecture was recently announced by Burungale, Skinner, Tian, and Wan \cite[Theorem 1.2]{BSTW}.
\end{remark}

We henceforth assume \eqref{mcit} and write $\lambda_p^\pm$, $\mu_p^\pm$ to mean either algebraic or analytic invariants. Letting $r_E$ denote the Mordell-Weil rank of $E$, Kundu and Ray conjecture \cite[Conjecture 3.17]{KR}  that $\lambda_p^\pm=r_E$ on a density 1 set of good supersingular primes. Assuming this conjecture, the condition on $\lambda$-invariants in Theorem \ref{main} could be removed and one would have the bound $\mu_p^\pm \leq 1$ on a density 1 set of good supersingular primes for $E$. 

We remark that if $r_E=0$ then \cite[Theorem 3.8]{KR} implies that both $\lambda_p^\pm$ and $\mu_p^\pm$ vanish for all but finitely many primes $p$ (in fact, $\Sel^\pm(E/\Qcyc)=0$ for these primes). Thus, the primary contribution of Theorem \ref{main} is in providing support for the vanishing of $\mu_p^\pm$ in the positive rank case.
 We also note that, under some mild assumptions, it is known \cite[Lemma 3.3]{KR} that $\lambda_p^\pm \geq r_E$ for all good supersingular primes $p>2$, thus the cases where $\ell<r_E$ in Theorem \ref{main} are mostly vacuous.

The crux of Chakravarthy's proof in the ordinary setting is constructing a bound (which holds for all but finitely many $p$) on the size of the modular symbols defining the ordinary $p$-adic $L$-function. In the supersingular case, the signed $p$-adic $L$-functions are defined via a decomposition theorem of Pollack \cite[Theorem 5.6]{PollackThesis}, and in particular they are not as immediately understood in terms of modular symbols.
The approach taken here is to instead apply Chakravarthy's bound to the sequence of Mazur-Tate elements $\theta_n$ for $E$ (which are defined using modular symbols), where one can show (see Proposition \ref{MTbound}) that there exists an integer $n_0$ such that for all but finitely many primes $p$,
\begin{equation}\label{MTboundintro}
\mu(\theta_n)\leq 1, \quad \text{for all $n\geq n_0$.}
\end{equation}
The lower bound $n_0$ depends only on the conductor of $E$ (and not on $p$). We then relate the Iwasawa invariants of the Mazur-Tate elements to those of the signed $p$-adic $L$-functions in order to deduce Theorem \ref{main}. 

The assumption on $\lambda$-invariants in Theorem \ref{main} comes from the fact that, while one can show that $\mu(\theta_n)=\mu_p^*$ for $n$ large enough of fixed parity (see Proposition \ref{thetaLp}), in this case the lower bound on $n$ in the asymptotic depends on both $p$ and $\lambda_p^\pm$. The idea is that if we assume $\lambda_p^\pm$ does not vary with $p$ then it is possible to take $p\gg0$ so that $\mu(\theta_n)=\mu_p^*$ holds for \emph{any} fixed $n$, and in particular for the $n_0$ appearing in  \eqref{MTboundintro}.

\subsection{Acknowledgements} We are grateful to A. Chakravarthy, J. Hatley, A. Lei, and the anonymous referee for their helpful comments in the preparation of this article. 

\section{Iwasawa invariants}

Fix a prime $p$ and let $F$ be a nonzero power series in $\Lambda=\Z_p\lb T\rb$. By the Weierstrass preparation theorem \cite[Theorem 7.3]{Washington} there are unique nonnegative integers $\lambda$ and $\mu$ such that
\begin{equation}\label{wdecomp}
F=p^{\mu}DU,
\end{equation}
for some distinguished polynomial $D\in \Z_p[T]$ of degree $\lambda$ and some $U\in\Lambda^\times$. (Recall that a polynomial $D\in \Z_p[T]$ is called \emph{distinguished} if $D\equiv T^{\deg D}\Mod p$.) In terms of the coefficients of $F=\sum_{i\geq 0} a_iT^i$, we have
\begin{align*}
\mu &= \min\{\ord_p (a_i)\mid i\geq 0\},\\
\lambda &= \min \{\, i \, \mid \, \ord_p (a_i)=\mu\}.
\end{align*}

\subsection{Refined Iwasawa invariants}
Let $\Gamma\cong \Z_p$ be the Galois group of the cyclotomic $\Z_p$-extension of $\Q$ and let $\Gamma_n=\Gamma/\Gamma^{p^n}\cong\Z/p^n\Z$ denote the Galois group of its $n$th-layer.  Let $\omega_n=(1+T)^{p^n}-1$ and $\Lambda_n=\Lambda/(\omega_n)$. Fixing a topological generator $\gamma\in\Gamma$, one has isomorphisms $\Lambda\cong \Z_p\lb \Gamma\rb$ and $ \Lambda_n\cong \Z_p[\Gamma_n]$ induced by the map  $\gamma\mapsto 1+T$. \emph{Refined} Iwasawa invariants are those attached to elements of  $ \Lambda_n$. We now give two definitions of refined Iwasawa  invariants -- both useful in different contexts -- and then show that they are equivalent. 

\subsubsection{Definition via the division algorithm} \label{refined1}
 Since $\Lambda=\displaystyle\lim_{\leftarrow} \Lambda_n$, for each $n\geq 0$ there is a projection map $\pi_n:\Lambda\twoheadrightarrow \Lambda_n,\, F\mapsto F\Mod \omega_n,$
and we can define the Iwasawa invariants of $\pi_n(F)$ as follows. Since $\omega_n$ is a distinguished polynomial, the division algorithm for distinguished polynomials in $\Lambda$ allows us to write 
$$
F=\omega_nQ_n+F_n,
$$
for some unique $Q_n\in \Lambda$ and a polynomial $F_n\in \Z_p[T]$ of degree $<p^n$.
Define 
\begin{align*}
\lambda(\pi_n(F))&=\lambda(F_n),\\\mu(\pi_n(F))&=\mu(F_n).
\end{align*}

\subsubsection{Definition via augmentation ideals}\label{refined2} 

Following \cite{PollackK} and \cite{PW}, one can define the Iwasawa invariants of $\theta\in \Z_p[\Gamma_n]$ as follows. For each $n\geq 1$, the element $\gamma_n= \gamma\Mod \Gamma^{p^n}$ generates $\Gamma_n$ and we define the $\mu$-invariant of $\theta=\sum_{j=0}^{p^n-1} c_j\gamma_n^j$ by 
$$
\mu(\theta)=\min_{0\leq j\leq p^{n-1}} \ord_p(c_j).
$$
For the $\lambda$-invariant, let $\theta'=p^{-\mu(\theta)}\theta\in \Z_p[\Gamma_n]$ and let $I_n$ be the augmentation ideal of $\F_p[\Gamma_n]$. (Thus, $I_n$ is the ideal generated by the image of $\gamma_n-1$ in $\F_p[\Gamma_n]$.) Since $\theta'$ has nonzero image under the natural reduction map $\overline{(\cdot)}:\Z_p[\Gamma_n]\rightarrow \F_p[\Gamma_n]$ and all ideals of $\F_p[\Gamma_n]$ are powers of $I_n$, we can define
$$
\lambda(\theta)=\ord_{I_n}\overline{\theta'}=\max\{j \mid \overline{\theta'}\in I_n^j\}\in \{0,1,\dots,p^n-1\}.
$$
(If $n=0$ then $\theta\in \Z_p$ and we define $\mu(\theta)=\ord_p(\theta)$  and $\lambda(\theta)=0$.)

\subsubsection{Equivalence of definitions}  We now show that the definitions of refined Iwasawa invariants given above agree. 
\begin{proposition}\label{IwInvEquiv} Let $n\geq 0$ and $\theta\in \Z_p[\Gamma_n]$. If $F\in \Z_p[T]$ is the unique polynomial of degree $<p^n$ mapping to $\theta$ under the composition 
$$
\Lambda \twoheadrightarrow \Lambda_n\xrightarrow{\cong}\Z_p[\Gamma_n], \quad T\mapsto \gamma_n-1, 
$$
then $\lambda(\theta)=\lambda(F)$ and $\mu(\theta)=\mu(F)$.  In particular, the Iwasawa invariants defined in \S\ref{refined1} and \S\ref{refined2} agree.  
\end{proposition}
\begin{proof} The case $n=0$ is clear, so suppose $n\geq 1$. Write $\theta=p^{\mu(\theta)}\theta' $ for some $\theta'\in \Z_p[\Gamma_n]$. Let $F_{\theta'}\in \Z_p[T]$ be a representative of the image of $\theta'$ in $ \Lambda_n$, so $F\equiv p^{\mu(\theta)}F_{\theta'}\Mod \omega_n$. By the division algorithm, we can choose $F_{\theta'}$ such that $\deg F_{\theta'}<p^n$, in which case the degree of $F$ forces the equality $F=p^{\mu(\theta)}F_{\theta'}$ in $\Z_p[T]$. Since $\theta'\equiv (\gamma_n-1)^{\lambda(\theta)}\theta''\Mod p$ for some $\theta''\in \Z_p[\Gamma_n]$, the commutativity of the diagram 
$$
\label{split}
\begin{tikzcd}
\Z_p[\Gamma_{n}] \arrow[r," "] \arrow[d," "] &\Lambda_n\arrow[d," "] \\
\F_p[\Gamma_{n}]\arrow[r," "] & \F_p[T]/(\omega_n),
 \end{tikzcd}
$$
where the horizontal maps are $\gamma_n\mapsto 1+T$ and the vertical maps are reduction mod $p$, implies that $F_{\theta'}\equiv T^{\lambda(\theta)}F_{\theta''}\mod p$ for some $F_{\theta''}\in \Oo[T]$ of degree $<p^n$. Hensel's lemma \cite[\S II Lemma 4.6]{Neukirch} now gives a factorization 
$F_{\theta'}=DU$ in $\Z_p[T]$, where $D,U\in \Oo[T]$ are such that $\deg  D=\lambda(\theta)$ and $D\equiv T^{\lambda(\theta)} \Mod p$ (so $D$ is distinguished), and $U\equiv F_{\theta''}\mod p$ (so $U\in \Lambda^\times$ since the constant term of $F_{\theta''}$ does not vanish mod $p$ by maximality of $\lambda(\theta)$). It follows that $F=p^{\mu(\theta)}DU$ and by uniqueness of the Weierstrass decomposition \eqref{wdecomp}, we obtain $\mu(\theta)=\mu(F)$ and $\lambda(\theta)=\lambda(F)$.
\end{proof}

\subsubsection{Relating invariants in $\Lambda$ and $\Lambda_n$}

The following lemma is known in the literature (see \cite[Remark 4.3]{PollackK}), though we outline a proof for completeness. 

\begin{lemma} \label{projprop} Let $n\geq 0$ and $F\in \Lambda$. If $\lambda(F)<p^n$ then the Iwasawa invariants of $F$ and $\pi_n(F)$ agree.
 \end{lemma}
\begin{proof} We may assume $\mu(F)=0$. Use the division algorithm to write \begin{equation}\label{barF}
F=\omega_nQ_n+F_n.
\end{equation}
If $\mu(F_n)$ were positive then \eqref{barF} implies 
$$
F\equiv T^{p^n}Q_n \Mod p,
$$ 
which contradicts the fact that $\lambda(F)<p^n$. Hence we must have $\mu(F_n)=0$. From \eqref{wdecomp}, we can write $F=(T^{\lambda(F)}+ pF_0)U$ for some $F_0\in \Z_p[T]$ of degree $<\lambda(F)$. Combining this decomposition with \eqref{barF} yields
$$
F_n\equiv T^{\lambda(F)}(U-T^{p^n-\lambda(F)}Q_n)\Mod p,
$$
but since $U$ is a unit and $\lambda(F)<p^n$, $U-T^{p^n-\lambda(F)}Q_n$ must also be a unit. It follows that $\lambda(F_n)=\lambda(F)$.
\end{proof}

\section{Bounding the $\mu$-invariant}

Let $E/\Q$ be an elliptic curve of conductor $N_E$ and fix a prime $p$ of good reduction such that $a_p=0$. 

\subsection{Mazur-Tate elements} Let $L_p^\pm\in \Lambda$ and $\theta_n\in \Q[T]$ denote the plus/minus $p$-adic $L$-functions and Mazur-Tate elements for $E$, as defined in \S\S 2.9 and 6.15 of \cite{PollackThesis}, respectively. The definition of both $L_p^\pm$  and  $\theta_n$ depend on a choice of complex periods $\Omega_E^\pm \in \C$. We henceforth assume  that $\Omega_E^\pm$ are $p$-cohomological periods for $E$, in the sense of \cite[\S 2.2]{PW}. The  choice of cohomological periods ensures that the coefficients of $\theta_n$ are $p$-integral (cf. \cite[Remark 2.2]{PW}), thus we can view each $\theta_n$ as an element of the localization $\Z_{(p)}[T]\subseteq \Q[T]$. 

We now relate the Iwasawa invariants of $\theta_n$ to those of $L_p^\pm$ by showing that the \emph{even} Mazur-Tate elements recover the \emph{minus}  invariants of the $p$-adic $L$-function, and vice-versa. Let $q_1=q_0=0$ and define for $n\geq 2$ the sequence
$$
q_n=\begin{cases}p^{n-1}-p^{n-2}+\cdots +p-1 & \quad\text{$n$ even,}\\
p^{n-1}-p^{n-2}+\cdots +p^2-p & \quad\text{$n$ odd.}
\end{cases}
$$
Let $\lambda_p^\pm$ and $\mu_p^\pm$ denote the Iwasawa invariant of $L_p^\pm$.

\begin{proposition}\label{thetaLp} 
If $n\geq 0$ is even (resp., odd) and $\lambda_p^{-}<p^n-q_n$ (resp., $\lambda_p^{+}<p^n-q_n$), then 
\begin{align*}
\mu(\theta_n) &= \mu_p^\pm,\\\lambda(\theta_n)&=\lambda_p^\pm +q_n,
\end{align*}
where $\pm$ is opposite the parity of $n$. 
\end{proposition}
\begin{proof} This follows from the argument of \cite[page 3]{RGLthesis}, which we reproduce in brief here. Let $\e_n=\sgn(-1)^n$ denote the parity of $n$. By \cite[Proposition 6.18]{PollackThesis}, we have
 \begin{equation}\label{poleq1}
\theta_n \equiv \omega_n^{-\e_n}L_p^{-\e_n}\Mod \omega_n.
\end{equation}
Here 
$$
\omega_n^+=\prod_{\underset{\text{$i$ even}}{1\leq i\leq n}}\Phi_{p^i}(1+T)\qquad \text{and}\qquad \omega_n^-=\prod_{\underset{\text{$i$ odd}}{1\leq i\leq n}}\Phi_{p^i}(1+T),
$$
where $\Phi_{p^i}(T)$ is the $p^i$th cyclotomic polynomial. Since $\lambda(\Phi_{p^n}(1+T))=p^n-p^{n-1}$, we have $\lambda\big(\omega_{n}^{-\e_n}L_p^{-\e_n}\big)=q_n+\lambda_p^{-\e_n}.$ As the sequence $p^n-q_n$ tends to infinity, we may therefore take $n$ large enough so that $\lambda\big(\omega_{n}^{-\e_n}L_p^{-\e_n}\big)<p^n$. The result now follows from Lemma \ref{projprop}.\end{proof}

\begin{remark}\nf  The formula for $\lambda$-invariants in Proposition \ref{thetaLp} can also be found in \cite[\S 5]{PR}, \cite[Theorem 4.1]{PW}, and \cite[Corollary 8.9]{Sprung}, where it is instead deduced from a 3-term compatibility relation (see \cite[Proposition 2.5]{PW}, for example) satisfied by Mazur-Tate elements. 
\end{remark}

We now fix an integer $\ell\geq 0$ and let $X^\pm(E,\ell)$ denote the set of all good supersingular primes $p>3$ for which $\lambda_{p}^\pm=\ell$. 

\begin{remark}\nf It is conjectured \cite[Conjecture 3.17]{KR} that $\lambda_p^\pm$ coincides with the Mordell-Weil rank on a density 1 set of good supersingular primes, thus one expects $X^\pm(E,\ell)$ to have density 0 except when $\ell=r_E$. 
\end{remark}

\begin{corollary}\label{coromu} Fix $n\geq 1$. For all but finitely many $p\in X^\pm(E,\ell)$, we have  
\begin{align*}
\mu_p^\pm&=\mu(\theta_n) 
\end{align*}
 where $\pm$ is opposite the parity of $n$.
\end{corollary}
\begin{proof}
Since $\lambda_p^\pm= \ell$, we can take $p$ large enough so that $\lambda_p^\pm<p^n-q_n$.
For such $p$, Proposition \ref{thetaLp} yields the desired result.
\end{proof}

\subsection{Modular symbols}
Let $f$ be the cuspidal newform attached to $E$ via modularity. For $r\in\Q$, recall the modular symbols of \cite{MT} defined by 
$$
[r]^\pm=\frac{\pi i}{\Omega_E^\pm}\bigg(\int_{r}^{i\infty}f(z)dz\pm\int_{-r}^{i\infty}f(z)dz\bigg).
$$
For  odd primes $p$ we have by definition (see \cite[Definition 6.15]{PollackThesis})
$$
\theta_n=\sum_{a\in (\Z/p^{n+1}\Z)^\times}[a/p^{n+1}]^+\gamma_n^{\log_\gamma a}\in \Z_p[\Gamma_n].
$$
Here $0\leq \log_\gamma(a)\leq p^{n}-1$ is the unique integer for which $a\equiv \omega(a)(1+p)^{ \log_\gamma(a)}\Mod p^{n+1}$, where $\omega:(\Z/p^{n+1}\Z)^\times\rightarrow \Z_p^\times$ is the mod $p$ cyclotomic character (sending $a\Mod p^{n+1}$ to the $(p-1)$st root of unity $\omega(a)\in \Z_p^{\times}$ with $\omega(a)\equiv a\Mod p$.) Write $\bm\mu_{p-1}$ for the set of $(p-1)$st roots of unity in $\Z_p^\times$, and if $\alpha\in \Z_p$ define $[\alpha/p^{n}]^\pm=[a/p^{n}]^\pm$ where $\alpha\equiv a\Mod p^n$. 

\begin{lemma}\label{thetamu} Let $p$ be an odd prime of good reduction. For any $n\geq 0$ we have
$$
\mu(\theta_n)=\displaystyle\min_{0\leq j\leq p^{n}-1} \ord_p\bigg(\sum_{\eta\in \bm\mu_{p-1}}\bigg[\frac{\eta (1+p)^{j}}{p^{n+1}}\bigg]^+\bigg).
$$
\end{lemma}
\begin{proof}Since $\log_\gamma(a)=\log_\gamma(b)$ if and only if $a\omega(a)^{-1}=b\omega(b)^{-1}\Mod p^{n+1}$, we have 
 \begin{equation}\label{MT}
\theta_n=\sum_{j=0}^{p^n-1} \sum_{\eta\in \bm\mu_{p-1}}\bigg[\frac{\eta (1+p)^{j}}{p^{n+1}}\bigg]^+\gamma_n^{j}.
\end{equation}
The result now follows from Definition \ref{refined2}. \end{proof}

\begin{lemma}\label{MSbound} Let $n\geq 0$ and $C\in \R$. For all but finitely many primes $p$,  if $a\in (\Z/p^{n+1}\Z)^\times$ then 
$$
\bigg|C\bigg[\frac{a}{p^{n+1}}\bigg]^+\bigg|<p. 
$$

\end{lemma}
\begin{proof}By Chakravarthy's bound \cite[Proposition 4.1]{Chakravarthy}, there are constants $c_1$ and $c_2$ depending only on the conductor of $E$ (and not $p$) such that for any $x\in \Q$, 
$$
|[x]^+|\leq c_1+c_2\log(\text{denominator}(x)).
$$
The result now follows by letting $x=a/p^{n+1}$ and taking $p$ large enough so that $c_1+c_2\log(p^{n+1})<\frac{p}{|C|}$. 
\end{proof}

\begin{lemma}\label{nonvanish} Let $p$ be an odd prime of good reduction. There is a constant $n_0$ depending only on $N_E$ such that if $n\geq n_0$ then
$$
\sum_{\eta\in \bm\mu_{p-1}}\bigg[\frac{\eta (1+p)^{j}}{p^{n+1}}\bigg]^+\neq0
$$
for some $0\leq j\leq p^{n}-1$. 
\end{lemma}
\begin{proof} Since $p$ is a good prime, a result of Chinta \cite[Theorem 2]{Chinta} guarantees the existence of an integer $n_0$ (depending only on $N_E$ and not on $p$) such that if $n\geq n_0$ and $\chi$ is a Dirichlet character of conductor $p^n$ then $L(E,\chi,1)\neq 0$. Let $\chi$ be an even Dirichlet character of $p$-power order and conductor $p^n$ with $n\geq n_0$. Setting $v=1+p$, we now have 
\begin{align*}
\sum_{j=0}^{p^n-1} \sum_{\eta\in \bm\mu_{p-1}}\chi(\eta v^j)\bigg[\frac{\eta v^{j}}{p^{n+1}}\bigg]^+&=\sum_{a\in (\Z/p^{n+1}\Z)^\times}\chi(a)\bigg[ \frac{a}{p^{n+1}}
\bigg]^+\\
&=\tau(\bar \chi)\frac{L(E,\chi,1)}{\Omega_E^+}\\
&\neq 0.
\end{align*}
Here $\tau(\chi)$ is a Gauss sum and the middle equality above is due to \cite[(8.6)]{MTT}. It follows that there is some $0\leq j\leq p^{n}-1$ for which 
$$
\sum_{\eta\in \bm\mu_{p-1}}\chi(\eta v^j)\bigg[\frac{\eta v^{j}}{p^{n+1}}\bigg]^+=\chi(v^j)\sum_{\eta\in \bm\mu_{p-1}}\bigg[\frac{\eta v^{j}}{p^{n+1}}\bigg]^+\neq 0,
$$
where the middle equality follows from the fact that $\chi$ has $p$-power order and $\eta$ is a $(p-1)$st root of unity. The result follows. 
\end{proof}

\subsection{Main result} We now prove our main theorem. First, we give a bound on the $\mu$-invariants of Mazur-Tate elements. 

\begin{proposition}\label{MTbound} There is a constant $n_0$ depending only on $N_E$ such that if $n\geq n_0$ then $\mu(\theta_n)\leq 1$ for all but finitely many primes $p$.
\end{proposition}
\begin{proof} Let $n_0$ be as in Lemma \ref{nonvanish} and take $n\geq n_0$, so that
\begin{equation}\label{MTcoef}
\sum_{\eta\in \bm\mu_{p-1}}\bigg[\frac{\eta (1+p)^{j}}{p^{n+1}}\bigg]^+\neq0
\end{equation}
holds for all good primes $p>2$ and some $0\leq j \leq p^{n}-1$. 
Note the sum in \eqref{MTcoef} is the $j$th coefficient of $\theta_n$ when written in the form \eqref{MT}. In particular, since $\theta_n\in \Z_{(p)}[T]$, this sum is a rational number whose denominator is $d_n$ not divisible by $p$. Thus $d_n \sum_{\eta\in \bm\mu_{p-1}}[\eta (1+p)^{j}/{p^{n+1}}]^+$ is a nonzero integer, and from Lemma \ref{MSbound} we can take $p$ large enough so that 
$$
\bigg|d_n \sum_{\eta\in \bm\mu_{p-1}}\bigg[\frac{\eta (1+p)^{j}}{p^{n+1}}\bigg]^+\bigg|\leq \sum_{\eta\in \bm\mu_{p-1}}\bigg|d_n\bigg[\frac{\eta (1+p)^{j}}{p^{n+1}}\bigg]^+\bigg|<(p-1)p<p^2.
$$
It now follows that 
$$
\ord_p\bigg(\sum_{\eta\in \bm\mu_{p-1}}\bigg[\frac{\eta (1+p)^{j}}{p^{n+1}}\bigg]^+\bigg)=\ord_p\bigg(d_n \sum_{\eta\in \bm\mu_{p-1}}\bigg[\frac{\eta (1+p)^{j}}{p^{n+1}}\bigg]^+\bigg)\leq 1.
$$
From Lemma \ref{thetamu}, we now have that $\mu(\theta_{n})\leq 1$.
\end{proof}

\begin{remark}\nf It is interesting to note that the bound in Proposition \ref{MTbound} applies to both ordinary and supersingular primes. In particular, if $p$ is an ordinary prime then by \cite[(4)]{PW} we have $\mu_p=\mu(\theta_n(f_\alpha))$ for $n\gg0$, where $\mu_p=\mu(L_{p}^{\an})$ and $f_\alpha$ is the $p$-stabilization of $f$ to level $pN$ at a root $\alpha$ of the Hecke polynomial $X^2-a_pX+p$. It is therefore tempting to try to deduce Chakravarthy's result \cite[Theorem 1.3]{Chakravarthy} that $\mu_p\leq 1$ for all but finitely many $p$ from Proposition \ref{MTbound}, however this does not immediately follow since the Iwasawa invariants of $\theta_n(E)$ and $\theta_n(f_\alpha)$ need not always agree (see \cite[Example 3.4]{PW}). 
\end{remark}

\noindent \it Proof of Theorem \ref{main}. \nf Fix $\ell\geq 0$. It suffices to show that for all but finitely many  primes $p\in X^\pm(E,\ell)$, we have $\mu_{p}^\pm\leq 1$. By Proposition \ref{MTbound}, there exists an odd integer $n^+$ and an even integer $n^-$, neither of which depends on $p$, such that  $\mu(\theta_{n^\pm})\leq 1$ holds for all but finitely many $p$.  But by Corollary \ref{coromu}, for either choice of sign $*\in \{+,-\}$ we have $\mu_p^*=\mu(\theta_{n^*})$ for all but finitely many $p\in X^*(E,\ell)$.  
\hfill $\square$

\end{document}